\documentclass{article}
\RequirePackage{amsmath}
\RequirePackage{amsthm}
\RequirePackage{amsfonts}
\RequirePackage{amssymb}

\newtheorem{definition}{Definition}
\newtheorem{theorem}{Theorem}
\newtheorem{lemma}{Lemma}
\newtheorem{corollary}{Corollary}

\begin{document}

\title{Conjecture on Supersequence Lower Bound related to Connell Sequence}
\author{Oliver Tan}
\date{December 2024}
\maketitle

\begin{abstract}
This paper proves the minimum size of a supersequence over a set of eight elements is $52$. This disproves a conjecture that the lower bound of the supersequence is the partial sum of the geometric Connell sequence. By studying the internal distribution of individual elements within substrings of the supersequence called segments, the proof provides important results on the internal structure that could help to understand the general lower bound problem for finite sets. 
\end{abstract}

A supersequence over a finite set is a sequence that contains as subsequence all permutations of the set. 
For example, $\langle 1,2,3,1,2,1,3 \rangle$ is a supersequence over $\{1,2,3\}$. It also happens to be the shortest supersequence over the set.

Starting with the arithmetic sequence $\langle 1,2,3,4,5... \rangle$, the arithmetic Connell sequence is constructed by alternatively taking a number of the next smallest odd or even integers. The number of integers taken is based on the arithmetic sequence. So, the arithmetic Connell sequence starts with one odd integer $1$, and then two even integers $2$, $4$, and then three odd integers $5$, $7$, $9$, and then the next four even integers $10$, $12$, $14$, $16$ and so on. The geometric Connell sequence starts with the geometric sequence $\langle 1,2,4,8,16... \rangle$. The first few integers are given by $\langle 1, 2, 4, 5, 7, 9, 11, 12, 14, 16, 18, 20, 22, 24, 26 \rangle$.

When the concept of supersequences was studied, Newey \cite{Newey} proposed two conjectures as the possible theoretical lower bound length. The first is based on the length of the shortest known supersequence at that time, namely $n^2-2n+4$. The second conjecture is defined by recurrence formulas, which happens to be the partial sums of the geometric Connell sequence. The shortest supersequences of the first seven smallest finite sets are as predicted by both conjectures. In 2011, however, the first conjecture was disproved by Zalinescu \cite{Zalinescu} and Radomirovic's \cite{Radomirovic} results. This paper proves the second conjecture was false too, and the counterexample is the supersequence over a set of eight letters. We prove that the minimum length is $52$. In contrast, the partial sum of the first eight elements of the geometric Connell sequence is $51$.

We utilize new abstraction tools like segmentation and reverse segmentation to divide a supersequence into manageable pieces. Based on the order of appearance of the letters from the finite set in the supersequence, each letter is given an identifier from $1$ to $8$. By defining a removal operation that deletes elements from a supersequence, we obtain the shortest supersequences over smaller sets where properties are known. From there, we derive rich information about individual letters and the number of occurrences, or frequencies of the letters within each piece of segment.

The general lower bound problem was studied in a 1975 paper by Kleitman-Kwiatkowski \cite{Kleitman} which claims that a supersequence over a set of size $n$ must have a minimum length of $n^2- C_{\epsilon} n^{7/4+\epsilon}$ for any $\epsilon > 0$, where $C_{\epsilon}$ is a constant depending on $\epsilon$. This result is clarified and improved by Plaxton-Tan \cite{Plaxton} in 2025 to a lower bound of $n^2 - Cn^{55/32}$, without the $\epsilon$.

For the upper bound, in 2022, Tan \cite{Tan} constructs the current shortest supersequence over a set of $n$ letters with length $\lceil n^2 - (\frac52 - \epsilon)n + C_\epsilon \rceil$. This improves on results by Radomirovic \cite{Radomirovic} of length $n^2 - \frac73n + \frac{19}3$, Zalinescu \cite{Zalinescu} of length $n^2 - 2m + 3$, and Newey \cite{Newey}, Adleman \cite{Adleman}, Koutas-Hu \cite{Koutas}, Galbiati \cite{Galbiati}, Mohanty \cite{Mohanty} of length $n^2 - 2m + 4$. It is still an open question what the optimal length of a supersequence over a set of $n$ letters is.

\section{Notations and Basic Results}

We follow most of the notations from Tan \cite{Tan}. When we say $\sigma $ is a sequence over a set $A$, it means that all elements of $\sigma $ are letters from the set $A$. The terminology distinguishes between an element and a letter. A member of a sequence is called an element, whereas a member of a set $A$ is called a letter. For example, given a sequence $\langle 1, 2, 4, 1, 3, 4\rangle$ over a set $A = \{1, 2, 3, 4\}$ and a subset $B = \{1, 2, 3\}$, we say the sequence has four elements and three letters from $B$. If $\sigma $ is a sequence, then $|\sigma|$ denotes the length of the sequence. The notation $f(\sigma, x)$ is called the frequency of $x$ in $\sigma$ and represents the number of occurrences of the letter $x$ within the sequence $\sigma$. So, for example, $f(\sigma, 1) = 2$ and $f(\sigma, 3) = 3$ where $\sigma = \langle 1, 2, 3, 1, 3, 4, 3 \rangle$. We also use multiple letters or a set of letters in the notation to represent the total number of occurrences of all of those letters, which means $f(\sigma, x_1, x_2, ..., x_i) = f(\sigma,  x_1) +f(\sigma,  x_2) + ... f(\sigma, x_i)$, or $f(\sigma, \{x_1, x_2, ..., x_i\}) = f(\sigma,  x_1) +f(\sigma,  x_2) + ... f(\sigma, x_i)$.

For any $i$ where $1 \le i \le |\sigma|$, the notation $\sigma [i]$ denotes the $i$th element of $\sigma $. Often, we abuse the notation to use a sequence $\sigma$ to represent the set of all letters occur in the sequence, i.e. $\sigma = \{a: a = \sigma[i]$ for some integer $i$ such that $1 \le i \le |\sigma|\}$. Accordingly, we can use set relation involving a sequence. For example, an element $a$ of $\sigma $ is denoted by $a \in \sigma $. The set relation $\sigma \subset A$ means $\sigma$ is a sequence over $A$.

The position $i$, where the element $a$ occurs in the sequence $\sigma $, is denoted by $\sigma ^{-1}[a]$. This definition is valid if $a$ occurs only once in the sequence $\sigma $. So when the definition is valid, we always have $\sigma [\sigma ^{-1}[a]] = a$. The last element, second last element and so on of $\sigma $ are denoted by $\sigma [-1]$, $\sigma [-2]$ and so on, respectively. Given two integers $i$ and $j$, let $\sigma [i,j]$ represent the substring of $\sigma $ starting at position $i$ and ending at $j$, inclusively for both ends. For example, if $\sigma = \langle 1,2,3,4,5,6 \rangle $, then $\sigma [3,-2] = \langle 3,4,5\rangle $ represents the substring from the third element to the second last element of $\sigma $. For any $i$ where $1 \le i \le |\sigma|$, the shorter notation $\sigma | i$ is also used to represent $\sigma [1,i]$ to denote the substring of $\sigma $ starting from $1$ and ending at $i$ inclusive at both ends.

Given a list of sequences $\sigma _1, \sigma _2, \sigma _3,...$, and letters $a_1, a_2, a_3,...$, then $\sigma _1 \sigma _2 \sigma _3...$ denotes the concatenation of those sequences $\sigma _1, \sigma _2, \sigma _3,...$, and $\sigma _1 a_1 \sigma _2 a_2 \sigma _3 a_3...$ denotes the concatenation of those sequences $\sigma _1, \sigma _2, \sigma _3,...$ interposed with those letters $a_1, a_2, a_3,...$. We may optionally write a dot $\cdot $ in between a sequence and an element, or between two sequences, if it makes reading easier, like $\sigma _1 \cdot a_1 \cdot \sigma _2 \cdot a_2 \cdot \sigma _3 \cdot a_3...$.

A $k$-\textit{perm} is a sequence of length $k$ where all elements of the sequence are distinct. For any integer $k$ less than the size of $A$, a sequence $\sigma $ is said to be \textit{$k$-complete over $A$} or simply \textit{$k$-complete}, if all the $k$-perms are subsequences of $\sigma$.


\section{Segmentation}

\begin{definition} \label{def1}
\textnormal {Let $\sigma_1, \sigma_2, ..., \sigma_n$ be a list of sequences over a set $A$ of size $n$. Suppose $k$ is an integer such that $1\le k \le n$, a $k$-perm $\rho$ over $A$ is said to be \textit{generated at $k$} if $k$ is the smallest integer such that $\rho$ is a subsequence of $\sigma_1 \cdot \sigma_2 \cdot ...\cdot \sigma_k$. More specifically, $\rho$ is said to be \textit{generated at $\sigma_k[i]$} for some integer $i \le |\sigma_k|$ if $\rho$ is a subsequence of $\sigma_1 \cdot \sigma_2 \cdot ...\cdot \sigma_k | i$, but not a subsequence of $\sigma_1 \cdot \sigma_2 \cdot ...\cdot \sigma_k | j$ for any integer $ 0 \le j < i$.}
\end{definition}

\begin{definition} \label{def2}
\textnormal {Suppose $\sigma$ is a supersequence over a set $A$ of size $n$. Let $\sigma_1, \sigma_2, ..., \sigma_n$ be a list of substrings of $\sigma$. The list is said to be a \textit{segmentation} of $\sigma$ if for any integer $k$ such that $1\le k \le n$, $\sigma_1 \sigma_2... \sigma_k$ is the minimum initial substring of $\sigma$ which is $k$-complete. In other words, $\sigma_1 \sigma_2... \sigma_k = \sigma | i$ for some integer $i$, and there is no integer $j < i$ such that $\sigma | j$ is $k$-complete. Each of the $\sigma_k$ is called a \textit{segment}. For any integers $i, j$ such that  $1 \le i < j \le n$, we use $\sigma_{i...j}$ to denote the concatenation $\sigma_i \cdot \sigma_{i+1} \cdot ... \cdot \sigma_{j-1} \cdot \sigma_j$}.
\end{definition}

It is not difficult to see that the segmentation of a supersequence is unique due to the minimum requirement.  For the rest of this paper, we assume $\sigma_1, \sigma_2, ..., \sigma_n$ is the segmentation of a supersequence $\sigma$ over a set $A$.

The reverse sequence of a supersequence is also a supersequence over the same set. Similar definitions can be designed to capture the reverse supersequence to help count the occurrences of letters.

\begin{definition} \label{def3}
\textnormal {Suppose $\sigma$ is a supersequence over a set $A$ of size $n$. Let $\rho_n, \rho_{n-1}, ..., \rho_1$ be a list of substrings of $\sigma$. The list is said to be a \textit{reverse segmentation} of $\sigma$ if for any integer $k$ such that $1\le k \le n$, $\rho_k \cdot ... \cdot \rho_2 \cdot \rho_1$ is the minimum final substring of $\sigma$ which is $k$-complete. In other words, $\rho_k \cdot ... \cdot \rho_2 \cdot \rho_1 = \sigma[i,-1]$ for some integer $i$, and there is no integer $j > i$ such that $\sigma[j,-1]$ is $k$-complete. Each of the $\rho_k$ is called a \textit{reverse segment}.}.
\end{definition}

The useful point to note at this time is that any result proven for segmentation equally holds true for the reverse segmentation with respect to the reverse of the supersequence.

\begin{lemma} \label{lem1} For all integers $k$ such that $1 \le k \le n$, $\sigma_k$ is not empty. 
\end{lemma}
\begin{proof} 
It is obvious that $\sigma_1$ is not empty. Suppose $k > 1$ is the smallest integer such that  $\sigma_k$ is empty. Let $a \in A$ be the last element of $\sigma_{k-1}$. Let $\sigma^{(1)}$ denote the sequence $\sigma_1 \cdot \sigma_2 \cdot ...\cdot\sigma_{k-1} [1,-2]$. In other words, $\sigma^{(1)}$ is the concatenation of $\sigma_1, \sigma_2, ..., \sigma_{k-1}$ with its last element removed.\\

\noindent  \textbf{Claim 1:} There exists a $(k-1)$-perm, with last element $a$, which is not a subsequence of $\sigma^{(1)}$.

If not, then every $(k-1)$-perm with last element $a$ is a subsequence of $\sigma^{(1)}$. By definition of $\sigma_{k-1}$, every $(k-1)$-perm with last element that is not $a$ is also a subsequence of $\sigma^{(1)}$.  Therefore, every $(k-1)$-perm is a subsequence of $\sigma^{(1)}$. This means that $\sigma^{(1)}$ is $(k-1)$-complete, which contradicts the assumption that $\sigma_1 \cdot \sigma_2 \cdot ...\cdot\sigma_{k-1}$ is the minimum substring of $\sigma$ that is $(k-1)$-complete.

Let $\rho$ be the $(k-1)$-perm given by Claim 1. Let $b$ be a letter of $A$ which does not occur in $\rho$. Then $\rho \cdot b$ is a $k$-perm, and yet it is not a subsequence of $\sigma_1 \cdot \sigma_2 \cdot ...\cdot\sigma_{k-1}$, contradicting the assumption that it is $k$-complete.
\end{proof}

The following lemma shows that the last element of any segment is a letter of $A$ that occurs only once in that segment.

\begin{lemma} \label{lem2}
Given any integer $k$ such that $1 \le k \le n$, if $a=\sigma_k[-1]$ is the last element in $\sigma_k$, then the letter $a$ has not occurred before $\sigma_k[-1]$ in $\sigma_k$.
\end{lemma}
\begin{proof} 
If $k=1$, then by definition, $\sigma_1[-1]$ is the last letter of $A$ that has its first occurrence in $\sigma_1$. So, $\sigma_1[-1]$ cannot also occur before the last element of $\sigma_1$. Suppose $k > 1$ and assume the letter $a=\sigma_k[-1]$ occurs also in $\sigma_k[i]$ for some $i < |\sigma_k|$. By definition of $\sigma_{k-1}$, any $(k-1)$-perm is a subsequence of $\sigma_1 \cdot \sigma_2 \cdot ...\cdot\sigma_{k-1}$.  Therefore, any  $k$-perm ending in the letter $a$ must be a subsequence of $\sigma_1 \cdot \sigma_2 \cdot ...\cdot\sigma_{k-1}\cdot \sigma_k |i$. This means that $a=\sigma_k[-1]$ is not needed in $\sigma_k$, since it contradicts the minimum requirement of $\sigma_k$.
\end{proof}

The order in which a letter first appears in a sequence is used to define the important concept of the last appearing letter. This concept is applied to different portions of a supersequence in later definitions to further identify individual letters.

\begin{definition} \label{def4}
\textnormal {Let $\sigma$ be a sequence over $A$. Suppose $B$ is a subset of $A$. The \textit{last appearing letter} of $B$ in $\sigma$ is a letter $a$ of $B$ which first appears last in $\sigma$, among all letter of $B$. In other words, $\sigma ^{-1}[a] > \sigma ^{-1}[b]$ for all $b \in B$.}
\end{definition}

\begin{definition} \label{def5}
\textnormal {Suppose $B$ is a subset of $A$. We use $L(\sigma, B)$ to denote the last appearing letter of $B$ in $\sigma$, i.e. $L(\sigma, B) = a$.}
\end{definition}

\begin{definition} \label{def6}
\textnormal {A \textit{duplicate in a segment} $\sigma_k$ is a letter $a \in \sigma_k$ which occurs more than once. In other words, there are integer positions $1 \le i < j \le |\sigma_k|$ such that $a = \sigma_k[i] = \sigma_k[j]$.}
\end{definition}

The definition of $L(\sigma, B)$ is valid even if there is a duplicate letter in $\sigma$. For example, if $\sigma = \langle 1, 2, 3, 4, 5, 4, 1, 3 \rangle$ and $B = \{1, 2, 4\}$, then $L(\sigma, B) = 4$. 

In the following definitions, we devise two different ways to label letters. The first way is to use bolded numbers to represent the last appearing letter within the supersequence. The second way is to define the last appearing letter of each segment as the terminal letter.

\begin{definition} \label{def7}
\textnormal {Given a sequence $\sigma$ over a set $A$ of size $n$, we use bolded numeric $\mathbf n$ to represent the last appearing letter $L(\sigma, A)$. Let $p_n$ be the position of the first occurrence of $\mathbf n$ in $\sigma$. For integers $1 \le k < n$, we use numeric $\mathbf {n-k}$ to represent the last appear letter $L(\sigma[p_{n-k+1} + 1,-1], A \setminus \{\mathbf n, \mathbf {n-1}, ... ,\mathbf {n-k+1}\})$. Let $p_{n-k}$ be the position of the first occurrence of $\mathbf {n-k}$ in $\sigma[p_{n-k+1} + 1,-1]$.}\end{definition}

\noindent \textbf{Example 1} Given $A =\{a, b, c\}$ and $\sigma = \langle a, b, a, c, a, c, b, a, c \rangle$.
\begin{enumerate}
  \item $\mathbf {3} = L(\sigma, A) = c$. The position of $\mathbf {3}$, or $c$ in $\sigma$ is $p_3 = 4$.
  \item $\mathbf {2} = L(\sigma[5,-1], A\setminus \{c\}) = b$. The position of $\mathbf {2}$, or $b$ is $p_2 = 7$.
  \item $\mathbf {1} = L(\sigma[8,-1], A\setminus \{c, b\}) = a$. The position of $\mathbf {1}$, or $a$ is $p_3 = 8$.
\end{enumerate}

\begin{definition} \label{def8}
\textnormal {Given a segmentation of a supersequence $\sigma_1, \sigma_2, ..., \sigma_n$ over a set $A$ of $n$ elements, define the \textit{terminal letter} of $\sigma_1$ to be the letter $L(\sigma_1, A)$.  For any integer $1 < k \le n$, let $B$ be the set of terminal letters of $\sigma_1$, $\sigma_2$, ..., $\sigma_{k-1}$. Then the terminal letter of $\sigma_k$ is defined to be the letter $L(\sigma_k, A \setminus B)$.}\end{definition}



\noindent \textbf{Example 2} Given $\sigma_1 = \langle 1,2,3,4,5,6,7,8 \rangle$,  $\sigma_2 = \langle 1,2,3,4,5,6,7 \rangle$, $\sigma_3 = \langle 1,8,2,3,4,5,6 \rangle$, and $\sigma_4 = \langle 1,8,2,3,4,5,7 \rangle$.
\begin{enumerate}
  \item The terminal letter of $\sigma_1$ is $8$, which is also represented as $\mathbf 8$.
  \item The terminal letter of $\sigma_2$ is $7$, which is also represented as $\mathbf 7$.
  \item The terminal letter of $\sigma_3$ is $6$, which is also represented as $\mathbf 6$.
  \item The terminal letter of $\sigma_4$ is $5$, which is also represented as $\mathbf 5$.
\end{enumerate}

\begin{definition} \label{def9}
\textnormal {The terminal letter of the first segment is also called the \textit{head terminal letter}. }
\end{definition}

The head terminal letter is always $\mathbf n$. In fact, the terminal letter of the first three segments $\sigma_i$, for $1 \le i \le 3$, is $\mathbf {n-i+1}$. This is however, not necessarily true for later segments.\\


An important technique to study occurrences of letters is to recursively reduce a supersequence to another one over a smaller set, where occurrences are known. The following is the method to do it.

\begin{definition} \label{def10}
\textnormal {Suppose $\sigma$ is a supersequence over a set $A$, and $\mathbf n = L(\sigma_1, A)$ is the head terminal letter. Define a \textit{removal operation} of $\mathbf n$ within $\sigma$ to be the removal of all elements of $\sigma_1$ and all occurrences of $\mathbf n$ within $\sigma$. The resulting sequence is denoted as $\sigma'$. }
\end{definition}

\begin{lemma} \label{lem3}
The removal operation of $\mathbf n$ within $\sigma$ produces a supersequence $\sigma'$ over $A \setminus \{\mathbf n\}$.
\end{lemma}
\begin{proof} 
If not, then let $\alpha$ be a permutation of  $A \setminus \{\mathbf n\}$ such that $\alpha$ is not a subsequence of $\sigma'$. Then the permutation $\mathbf n \cdot \alpha$ of $A$ is not a subsequence of $\sigma$, which is a contradiction. \end{proof}

\section{Results from Smaller Sets}

The following result by Newey \cite{Newey} will be needed when we apply the removal operation of the head terminal letter to obtain information about the minimum occurrences of each letter in a segment.

\begin{theorem} \label{th1}
The minimum length of a supersequence over a set of three, four, five, six and seven are $7$, $12$, $19$, $28$, $39$ respectively.
\end{theorem}

Some preliminary results about the length of the initial segments.

\begin{lemma} \label{lem4}
The minimum lengths of the first four segments of a minimum-sized supersequence over a set of $n$ letters are given by:
\begin{enumerate}
  \item $|\sigma_1| \ge n, |\sigma_2| \ge n-1$.
  \item If $n \in \sigma_2$, then  $|\sigma_2| \ge n$ and $|\sigma_3| \ge n-2.$ Otherwise $|\sigma_3| \ge n-1$.
  \item If $n-1 \in \sigma_3$, then $|\sigma_3| \ge n-1$ or $|\sigma_3| \ge n$, and $|\sigma_4| \ge n-3.$ Otherwise $|\sigma_4| \ge n-2$.
\end{enumerate} 
  Consequently, the minimum length of the substring of the first four segments is $|\sigma_{1...4}| \ge 4n-4$.

\end{lemma}
\begin{proof} The first item is by $1$ and $2$-completeness. For the third segment, if the head terminal does not occur in two consecutive segments, then $\sigma_1$ has to be $2$-complete for the $n-1$ letters from the set $A \setminus \{\mathbf n\}$. By recursive argument, that requires $(n-1) + (n-2) + 1$ length. This is not possible because a removal operation will remove too many elements which produces a supersequence smaller than the minimum. For the fourth segment, the same argument shows that $n-1$ has to occur either in $\sigma_3$ or $\sigma_4$.
\end{proof}

\noindent \textbf{Example 3} For a set of eight letters, the following are the possible configurations of the first four segments:
\begin{enumerate}
  \item $\sigma_1=\langle 1,2,3,4,5,6,7,8 \rangle$, $\sigma_2=\langle 1,2,3,4,5,6,7 \rangle$, $\sigma_3=\langle 1,8,2,3,4,5,6 \rangle$, $\sigma_4=\langle 1,7,2,3,4,5 \rangle$.
  \item $\sigma_1=\langle 1,2,3,4,5,6,7,8 \rangle$, $\sigma_2=\langle 1,2,3,4,5,6,8,7 \rangle$, $\sigma_3=\langle 1,2,3,4,5,6 \rangle$, $\sigma_4=\langle 1,7,2,3,4,5 \rangle$.
    \item $\sigma_1=\langle 1,2,3,4,5,6,7,8 \rangle$, $\sigma_2=\langle 1,2,3,4,5,6,8,7 \rangle$, $\sigma_3=\langle 1,2,3,4,5,7,6 \rangle$, $\sigma_4=\langle 1,2,3,4,5 \rangle$.
    \item $\sigma_1=\langle 1,2,3,4,5,6,7,8 \rangle$, $\sigma_2=\langle 1,2,3,4,5,6,7 \rangle$, $\sigma_3=\langle 1,8,2,3,4,5,7,6 \rangle$, $\sigma_4=\langle 1,2,3,4,5 \rangle$.
\end{enumerate}
In all configurations, $8$ occurs twice and $7$ occurs thrice in the four segments, and the total length of the segments is $28$.

The above size requirement is derived based on the definition of completeness. As will be seen in the later section, for the segmentation of a supersequence, the letter $8$ has to appear in $\sigma_4$ as well. \\

The following is the key result of occurrences of letters with smaller sets.

\begin{theorem} \label{th2}
The minimum number of occurrences of any letter in a minimum-sized supersequence over a set of seven or six letters is four. \end{theorem}
\begin{proof} 
The proof is similar for a set of seven letters, so we will prove the Theorem for a set of six. Suppose $\sigma$ is a supersequence over a set $A$ of six letters. If $\mathbf 6 = L(\sigma, A)$ occurs four times and another letter $x$ occurs three times, a removal operation of $\mathbf 6$ will produce a supersequence over a set of five letters where $x$ occurs only twice, which can be proven to be impossible. So without loss of generality, we assume $\mathbf 6$ occurs three times. The letter $\mathbf 6$ occurs in $\sigma_1$ and $\sigma_{5...6}$ due to the $1$-completeness of the first segment and first reverse segement. Performing a removal operation of $\mathbf 6$ will remove three occurrences of it plus the entire $\sigma_1$. But the total number of elements removed cannot be bigger than $28-19=9$, the difference between the minimum lengths of supersequences over a set of six and a set of five. So, the maximum length of $\sigma_1$ is seven. That means there is at most one duplicate element in $\sigma_1$.

If $\mathbf 6$ also occurs in $\sigma_2$, then it cannot occur in $\sigma_{3...4}$. This is not possible because it will require that $\sigma_2$ be $3$-complete for a set of five letters, and that requires $5+4+4$ elements, plus the two occurrences of $\mathbf 6$. This exceeds the maximum length of $\sigma_2$. So we assume $\mathbf 6 \in \sigma_3$ and $\mathbf6 \notin \sigma_2, \sigma_4$.

Let $\sigma^{(1)}$ be the supersequence before the occurrence of $\mathbf 6$ in $\sigma_3$, namely $\sigma^{(1)} = \sigma_{1...2} \cdot \sigma_3[1, \sigma_3^{-1}[\mathbf 6]]$. Let $\sigma^{(2)}$ be the rest of the supersequence. The five letters of $A \setminus \{\mathbf6\}$ have to be $3$-complete before the position of $\mathbf 6$ in $\sigma_3$. That requires $5+4+4$ elements and $|\sigma^{(1)}| \ge 15$. Note there are three occurrences of $\mathbf 1$ in $\sigma_{4...6}$ and one occurrence of $\mathbf 6$ in $\sigma^{(2)}$. Consider the following two cases.\\

\noindent \textbf{Case 1:} There is no duplicate in the first three segments. So, the length of $\sigma^{(2)}$ is at most $13$. Therefore, there is at most one letter from $\{\mathbf2,\mathbf3,\mathbf4,\mathbf5\}$ occurring three times in $\sigma^{(2)}$. Without loss of generality, suppose it is $\mathbf4$. The $3$-perm $\langle \mathbf4,\mathbf1,\mathbf6 \rangle$ or $\langle \mathbf 1,\mathbf4,\mathbf6 \rangle$ is generated at $\mathbf6$ in $\sigma_3$. But the rest of three letters $\{\mathbf2,\mathbf3,\mathbf5\}$ cannot form a supersequence over three letters in $\sigma^{(2)}$ because none of them occur thrice.\\

\noindent \textbf{Case 2:} There is a duplicate in the first three segments. So, the length of $\sigma^{(2)}$ is at most $12$. Therefore, there is no letter from $\{\mathbf2,\mathbf3,\mathbf4,\mathbf5\}$ occurring three times in $\sigma^{(2)}$. If the duplicate does not occur in $\sigma_1$ or if the duplicate is not $\mathbf1$, the same reasoning of Case 1 proves the Theorem. Otherwise, the $3$-perm $\langle \mathbf3,\mathbf2,\mathbf1 \rangle$ is generated at $3$ and the rest of the letters $\{\mathbf4,\mathbf5,\mathbf6\}$ all occur twice in $\sigma^{(2)}$, but that cannot form a supersequence over three letters.
\end{proof}
 
The first two items of the following lemma are proven by Newey \cite{Newey}
\begin{lemma} \label{lem5}
The frequency distributions of the letters of a minimum supersequence over sets smaller than eight are given by the following.
\begin{enumerate}
  \item For set of three letters, the frequencies are given by $4,2,1$; $3,3,1$ or $3,2,2$.
  \item For set of four letters, the frequencies are given by $4,3,3,2$.
  \item For set of five letters, the frequencies are given by $5,4,4,3,3$.
  \item For set of six letters, the frequencies are given by $6,5,5,4,4,4$.
  \item For set of seven letters, the frequencies are given by $7,6,6,5,5,5,5$ or $7,6,6,6,5,5,4$.
\end{enumerate} \end{lemma}
\begin{proof} Since the minimum frequency of any letter for a supersequence over a set of five is three, we can perform a removal to produce a supersequence over a set of four, which only has one possible distribution. So we conclude the distribution over a set of five as claimed. Likewise, the minimum frequency of any letter for a supersequence over a set of six is four, and we arrive at the claimed distribution. For a set of seven, since its letter minimum frequency can be four, there is an alternative distribution as given by the following example.
\end{proof}
 
\noindent \textbf{Example 4} The following is the segmentation of a non-Newey style supersequence over a set of seven letters with $f(\sigma, 7) = 4$.\\
$\sigma_1 = \langle 1,2,3,4,5,6,7 \rangle$. $\sigma_2 = \langle 1,2,3,4,5,6 \rangle$. $\sigma_3 = \langle 1,7, 2,3,4,5 \rangle$. $\sigma_4 = \langle 1,6, 2,3,7,4 \rangle$. $\sigma_5 = \langle 1,5,6,2,3 \rangle$. $\sigma_6 = \langle 1,4,7,5,6,2 \rangle$. $\sigma_7 = \langle 3,1,4 \rangle$.\\

\noindent \textbf{Example 5} The following are three minimum-sized supersequences over a set of three letters with the three possible frequency distributions.\\
$\langle a, b, a, c, a, b, a\rangle$. $\langle a, b, a, c, b, a, b\rangle$. $\langle a, b, c, a, b, a, c\rangle$.

\begin{lemma} \label{lem6}
The minimum lengths of the first five substrings of segments of a minimum-sized supersequence over a set of seven letters are given by:
\begin{enumerate}
  \item $|\sigma_1| \ge 7$, $|\sigma_{1...2}| \ge 13$, $|\sigma_{1...3}| \ge 19$, $|\sigma_{1...4}| \ge 24$.
  \item $|\sigma_{1...5}| \ge 29$.
\end{enumerate} \end{lemma}
\begin{proof} The minimum lengths of the first four substrings are given by Lemma 4. If $\sigma_{1...4}$ is of minimum length, then there are additional letters $\{\mathbf1,\mathbf2,\mathbf3\}$ to occur in $\sigma_5$. Also, since $\mathbf7 \notin \sigma_4$, so $\mathbf 7 \in \sigma_5$, because otherwise, $\mathbf 7$ will not occur in two consecutive segments. Similarly, either $\mathbf6$ or $\mathbf5$ must also occur in $\sigma_5$. So $|\sigma_5| \ge 5$. This proves that $|\sigma_{1...5}| \ge 29$.
\end{proof}

\section{The Frequencies of Each Letters}

For the rest of the paper, we assume $\sigma$ is a supersequence over a set of eight letters, and its length is $51$. We will derive a contradiction at the end of the paper.

The following lemma places an upper bound on the length of the segments based on the number of occurrences of their terminal letters.

\begin{lemma} \label{lem7} If $|\sigma| = 51$, then we have the following.
\begin{enumerate}
  \item The maximum possible value of $|\sigma_1|$ is $m_1 = 13 - f(\sigma, \mathbf8)$.
  \item The maximum possible value of $|\sigma_2|$ is $m_2 = 12 - f(\sigma, \mathbf7) + f(\sigma_1, \mathbf7)+ f(\sigma_2, \mathbf8) + m_1 - |\sigma_1|$.
  \item The maximum possible value of $|\sigma_3|$ is $m_3 = 10 - f(\sigma, \mathbf6)+ f(\sigma_{1...2}, \mathbf6) + f(\sigma_3, \mathbf8,\mathbf7) + m_2 - |\sigma_2|$.
  \item The maximum possible value of $|\sigma_4|$ is $m_4 = 8 - f(\sigma, \mathbf5) + f(\sigma_{1...3}, \mathbf5)+ f(\sigma_4, \mathbf8,\mathbf7,\mathbf6) + m_3 - |\sigma_3|$.
\end{enumerate} \end{lemma}
\begin{proof} 
For every removal operation of the head terminal letter $x$ performed, the number of elements removed from the supersequence is $f(\sigma, x) + |\sigma_1| -1$. This number cannot be bigger than the difference in length between the supersequence and the smallest length of the next supersequence as given by the series $39$, $28$, $19$ and $12$ from Theorem 1. Note also each removal operation of $x$ removes the head terminal letter from later segments, and reduces the number of count $f(\sigma, x)$ by one in the resulting supersequence.

\begin{enumerate}
  \item Since $f(\sigma, \mathbf 8) + |\sigma_1| -1 \le 51-39$. The maximum value of $|\sigma_1|$ is as stated.
  \item Denote the resulting supersequence from the first removal operation as $\sigma'$. Its length is given by $|\sigma'|=39+m_1 - |\sigma_1|$. The next removal operation gives $f(\sigma', 7) + |\sigma'_1| -1 \le |\sigma'|-28$, where $\sigma'_1$ is the first segment of $\sigma'$. The maximum value of $|\sigma'_1|$ is $12 - f(\sigma', 7) + m_1 - |\sigma_1|$. But  $f(\sigma', 7) = f(\sigma, 7) -  f(\sigma_1, 7) $, and $|\sigma_2| = |\sigma'_1| + f(\sigma_2, 8)$. The value of $m_2$ is as stated.
  \item The argument for the rest is similar.  \end{enumerate} \end{proof}

We start by counting the minimum frequencies of letters in the following lemma. Later, we will improve these counts.

\begin{lemma} \label{lem11}
The following are the minimum occurrences of each letter within $\sigma$. 
\begin{enumerate}
  \item $f(\sigma, \mathbf8) \ge 4$, $f(\sigma, \mathbf7) \ge 4$, $f(\sigma, \mathbf6) \ge 4$.
  \item $f(\sigma, \mathbf5) \ge 5$.
  \item $f(\sigma, \mathbf4) \ge 6$, $f(\sigma, \mathbf3) \ge 6$, $f(\sigma, \mathbf2) \ge 6$,  $f(\sigma, \mathbf1) \ge 6$.
\end{enumerate} \end{lemma}
\begin{proof} 
\begin{enumerate}
  \item These are corollaries of Theorem 2.
  \item By definition of $\mathbf 5$, we have $5 \in \sigma_i$, for all $1 \le i \le 4$. Also, $\mathbf 5 \in \sigma_{5...8}$ because $\mathbf5 \in \rho_1$ and $\rho_1$ is a substring of $\sigma_{5...8}$.
  \item Same reasoning as above.
\end{enumerate} \end{proof}

\begin{corollary} \label{corollary1}
The letters $\mathbf4$, $\mathbf3$, $\mathbf2$ and $\mathbf1$ cannot be the head terminal letter of the reverse segmentation.
\end{corollary}
\begin{proof} 
If any letters with at least six occurrences is the head terminal letter, then the maximum value of the length of the first segment is $13 - 6=7$. This contradicts the $1$-completeness definition.
\end{proof}

\begin{lemma} \label{lem9}
The maximum length of $\sigma_{1...4}$ is $31$.  \end{lemma}
\begin{proof} If $\alpha = \langle a, b, c, d \rangle$ is a $4$-perm generated at the last element of $\sigma_4$, then $\sigma_{5...8}$ must form a supersequence over the other four letters in $A \setminus \alpha$. So, that takes up at least twelve elements, meaning $f(\sigma_{5...8}, A \setminus \alpha) \ge 12$. It is possible to create another $4$-perm or $5$-perm using letters from $A \setminus \alpha$ and a letter, say $a$ from $\alpha$, that is generated at the same position. It shows that $f(\sigma_{5...8}, b,c,d) \ge 7$. So the total length of $\sigma_{5...8}$ is at least $20$, and $\sigma_{1...4} \le 31$. \end{proof}

\begin{lemma} \label{lem10}
If $\rho$ is the reverse sequence of $\sigma$, then the concatenation of its segments $\rho_2 \cdot \rho_1$ is a substring of $\sigma_{5...8}$. \end{lemma}
\begin{proof}  Both the maximum length for $\rho_1$ and $\rho_2$ is nine. If it is not a substring of $\sigma_{5...8}$, then the length of $\sigma_{5...8}$ is strictly smaller than $18$ while the length of $\sigma_{1...4}$ is larger than $33$. This contradicts the previous lemma. \end{proof}

Using the new results, we can improve on the counting of the occurrences of some letters.

\begin{lemma} \label{lem11}
The following are the minimum occurrences. 
\begin{enumerate}
  \item $f(\sigma, \mathbf7) \ge 5$, $f(\sigma, \mathbf6) \ge 5$.
  \item $f(\sigma, \mathbf5) \ge 6$.
\end{enumerate} \end{lemma}
\begin{proof} 
We utilize the fact that $\rho_2 \cdot \rho_1$ is a substring of $\sigma_{5...8}$ and divide the proof based on the head terminal letter of the reverse segmentation. \\
  
\noindent \textbf{Case 1:} If $\mathbf8$ is the head terminal letter of the reverse segmentation, then the letters $\mathbf7$, $\mathbf6$ and $\mathbf5$ must occur in both $\rho_2$ and $\rho_1$. Since $\sigma_{1...4}$ contains three occurrences of $\mathbf7$ and $\mathbf6$, and four occurrences of $\mathbf5$, the lemma is proven. \\
 
\noindent \textbf{Case 2:} If $\mathbf5$ is the head terminal letter of the reverse segmentation, then the same reasoning proves $f(\sigma, \mathbf7) \ge 5$, $f(\sigma, \mathbf6) \ge 5$. It remains to prove $f(\sigma, \mathbf5) \ge 6$.

Since $f(\sigma, \mathbf5) \ge 5$, the segment $\rho_1$ cannot contain duplicates. However, any $6$-perm generated at the last element of $\sigma_6$ will prove that $\sigma_{7...8}$ must contain a letter that occurs more than once. So $\rho_1$ must be strictly a substring of $\sigma_{7...8}$.

If $f(\sigma, \mathbf5) = 5$, then $\mathbf5$ does not occur in $\sigma_{5...6}$. So, any $5$-perm formed from the seven letters in $A \setminus \{\mathbf5\}$ must be generated before the last occurrence of $\mathbf5$ in $\sigma_4$. According to Lemma 6, that requires at least $29$ elements. Add to that the four occurrences of the letter $\mathbf5$, the length of $\sigma_{1...4}$ is at least $33$, which contradicts Lemma 9. \\

\noindent \textbf{Case 3:} Suppose either $\mathbf7$ or $\mathbf6$ is the head terminal letter of the reverse segmentation. Without loss of generality, assume it is $\mathbf7$. Then for the same reason as in previous cases, $f(\sigma, \mathbf6) \ge 5$ and $f(\sigma, \mathbf5) \ge 6$. It remains to prove $f(\sigma, \mathbf7) \ge 5$.

If $|\rho_1| = 9$, then a removal operation of $\mathbf7$ from the reverse segmentation will produce a supersequence over a set of seven where $\mathbf 8$ occurs $f(\sigma, \mathbf8)-1$ many times. By Theorem 2, we have $f(\sigma,\mathbf 8) \ge 5$. So a removal operation of $\mathbf8$ from the segmentation will prove $f(\sigma, \mathbf7) \ge 5$, as required.

So, assume $|\rho_1| = 8$ which contains no duplicate. The same reasoning as in Case 2 proves that $\rho_1$ must be strictly a substring of $\sigma_{7...8}$.

If $f(\sigma, \mathbf7) = 4$, then $\mathbf7$ does not occur in $\sigma_{5...6}$. So, any $5$-perm formed from the seven letters in $A \setminus \{\mathbf7\}$ must be generated before the last occurrence of $\mathbf7$ in $\sigma_4$. That requires at least $29$ elements according to Lemma 6. Add to that the three occurrences of $\mathbf7$, there are at least $32$ elements in $\sigma_{1...4}$. But this contradicts the maximum size of $\sigma_{1...4}$ as given by Lemma 9.
\end{proof}

\begin{corollary} \label{corollary3}
If $f(\sigma, \mathbf8) = 4$, then $\mathbf8$ is the head terminal letter of the reverse segmentation.
\end{corollary}
\begin{proof} 
If $\mathbf7$ or $\mathbf6$ is the head terminal letter of the reverse segmentation, a removal operation on it from the reverse supersequence will yield a supersequence over seven letters where $\mathbf8$ occurs only three times, contradicting Theorem 2.
\end{proof}

\section{The Head Terminal Letter}

With previous results of frequencies of each letter, we will now focus on the frequency of $\mathbf8$. Later in the section, we improve on the frequencies of $\mathbf7$, $\mathbf6$ and $\mathbf4$ to arrive at the conclusion of the paper.

\begin{lemma} \label{lem12}
If $f(\sigma, \mathbf8) = 4$, then $f(\sigma, \mathbf6) = 6$.\end{lemma}
\begin{proof} Suppose $f(\sigma, \mathbf8) = 4$. If $\mathbf8$ occurs in $\sigma_2$, then $\mathbf8$ can only occur once in the remaining four segments $\sigma_{3...6}$, implying it does not occur in two consecutive segments, which is impossible. So $\mathbf8 \notin \sigma_2$ and $\mathbf8 \in \sigma_3$. Also $\mathbf8 \notin \sigma_4$. The same reasoning shows that if $f(\sigma, \mathbf6) = 5$, it does not occur in $\mathbf6 \notin \sigma_4$.
Since both letters $\mathbf6$ and $\mathbf8$ do not occur in $\sigma_4$, it is possible to find a $4$-perm $\alpha$, generated at $4$, without using $\mathbf6$ and $\mathbf8$. One of the $6$-perms $\alpha \cdot \mathbf6 \cdot \mathbf8$ or $\alpha \cdot \mathbf8 \cdot \mathbf6$ is generated at the last occurrence of $\mathbf6$ or $\mathbf8$. This has to be in the first reverse segment $\rho_1$. So appending the remaining two letters from $A$ to the $6$-perm cannot be generated within $\sigma$, which is a contradiction. \end{proof}

The following is the core part of the entire proof.

\begin{lemma} \label{lem13}
The frequency of the letter $\mathbf8$ is $f(\sigma, \mathbf8) \ge 5$.
\end{lemma}
\begin{proof} Suppose $f(\sigma, \mathbf8) = 4$. Together with the results of the previous lemma, there are at most two duplicates in $\sigma_{1...4}$. Not counting the duplicates, the following table summarizes the minimum frequencies of all the letters in different segments.

\begin{table}[ht]
\caption{Frequency in Segments without the Two Duplicates} 
\centering 
\begin{tabular}{c c c c c c} 
\hline\hline 
Letter & $\sigma_{1...3}$ & $\sigma_4$ & $\sigma_{1...4}$ & $\sigma_{5...8}$ & $\sigma$ \\ [0.5ex] 
\hline 
$\mathbf8$ & 2 & 0 & 2 & 2 & 4 \\ 
$\mathbf7$ & 2 or 3 & 1 or 0 & 3 & 2 & 5\\
$\mathbf6$ & 3 & 1 or 0 & 4 or 3 & 2 or 3 & 6 \\
$\mathbf5$ & 3 & 1 & 4 & 2 & 6 \\
$\mathbf4$ & 3 & 1 & 4 & 2 & 6 \\
$\mathbf3$ & 3 & 1 & 4 & 3 & 7 \\
$\mathbf2$ & 3 & 1 & 4 & 3 & 7\\
$\mathbf1$ & 3 & 1 & 4 & 4 & 8\\ [1ex] 
\hline 
\end{tabular}
\label{table:nonlin} 
\end{table}

The proof of the lemma is based on the following series of Claims.\\

\noindent \textbf{Claim 1:}  The letter $\mathbf7$ does not occur in $\sigma_3$.
\begin{proof}  Assume the Claim is false. Pick any four letters $\{a, b, c, d\}$ from $A \setminus \{\mathbf7, \mathbf8\}$ that do not have duplicates in $\sigma_{1...3}$. Recall the frequency distribution of letters for supersequences over four letters is $4$, $3$, $3$, $2$. Since $f(\sigma_{1...3}, x) < 4$ for all letters $x \in \{a, b, c, d\}$, the set $\{a, b, c, d\}$ cannot form a supersequence over four letters in $\sigma_{1...3}$. So, there is a $4$-perm formed from letters in this set that is generated at, say position $i$ in $\sigma_4$. The rest of the letters from $A \setminus \{a, b, c, d\}$ must form a supersequence over four letters in $\sigma_4[i+1,-1] \cdot \sigma_{5...8}$. But this is not possible because $f(\sigma_{4...8}, \mathbf7) = 2$ and  $f(\sigma_{4...8},\mathbf 8) = 2$.  \end{proof}

\noindent \textbf{Claim 2:} The letter $\mathbf7$ occurs in $\sigma_4$.
\begin{proof} If the letter $\mathbf7$ does not occur in two consecutive segments $\sigma_3$ and $\sigma_4$, then $\sigma_{1...2}$ must have length at least $7 + 6 + 6 = 19$. Given there are only two duplicates, this is not possible. \end{proof}

\noindent \textbf{Claim 3:} The last element of $\sigma_4$ is either $\mathbf5$ or $\mathbf7$.
\begin{proof} By definition, all other letters occur before $\mathbf5$. \end{proof}

\noindent \textbf{Claim 4:} The letter $\mathbf8 \in \sigma_1, \sigma_3, \sigma_5, \sigma_{7...8}$.
\begin{proof} The letter $\mathbf8$ is obviously in the first and the last ones due to $1$-completeness of the first segment and first reverse segment. If $\mathbf8 \in \sigma_2$, then $\mathbf8$ can only occur once in the remaining four segments $\sigma_{3...6}$, implying it does not occur in two consecutive segments, which is impossible. Same reason it must occur in $\sigma_5$.   \end{proof}

\noindent \textbf{Claim 5:} The letter $\mathbf8$ occurs before $\mathbf5$ in $\sigma_3$.
\begin{proof} Assume the Claim is false. Let $a$, $b$ be two letters from $A \setminus \{\mathbf8,\mathbf7,\mathbf4, \sigma_4[-1]\}$ such that $\langle a, b \rangle$ is generated at $2$. Then the $4$-perm $\langle a, b, \mathbf8, \sigma_4[-1] \rangle$ is generated at the last position of $\sigma_4$. The rest of the letters must form a supersequence over four letters in $\sigma_{5...8}$, but this is not possible because they include letters $\mathbf4$ and either $\mathbf7$ or $\mathbf5$. All these letters have a frequency of two in $\sigma_{5...8}$. \end{proof}

\noindent \textbf{Claim 6:} There is at least one duplicate of a letter from $\{\mathbf5,\mathbf6,\mathbf7\}$ in $\sigma_{1...3}$.
\begin{proof} If not, then it is possible to find a $3$-perm formed by these letters, which is generated at the position of $\mathbf5$ in $\sigma_3$. Since $\mathbf8$ occurs before $\mathbf5$ in $\sigma_3$, appending $\mathbf8$ to this $3$-perm can't be generated at $4$, contradicting the $4$-completeness. \end{proof}

\noindent \textbf{Claim 7:} There is at least one duplicate of a letter from $\{\mathbf1,\mathbf2,\mathbf3\}$ in $\sigma_{1...3}$.
\begin{proof} If not, all letters from $\{\mathbf1,\mathbf2,\mathbf3,\mathbf7\}$ occur strictly less than four times in $\sigma_{1...3}$. So there is a $4$-perm formed from these letters that is generated at $4$. After that, the rest of letters $\{\mathbf4,\mathbf5,\mathbf6,\mathbf8\}$ must also form a supersequence over four letters. But this is not possible because all the letters $\{\mathbf4,\mathbf5,\mathbf6,\mathbf8\}$ also occur strictly less than four times in $\sigma_{4...8}$. \end{proof}

\noindent \textbf{Claim 8:} The position of the letter $\mathbf8$ in $\sigma_5$ is at least four, i.e. $\sigma_5^{-1}[\mathbf8] \ge 4$.
\begin{proof} It is possible to produce a $5$-perm generated at $5$ by avoiding the letter $\mathbf8$. Any such $5$-perm must be generated at a position before $\mathbf8$ in $\sigma_5$ or else appending $\mathbf8$ to it will yield a $6$-perm that can't be generated at $6$, violating the definition of $6$-completeness. Since there are at least three such $5$-perms having different last letters, the position of $8$ is at least four. \end{proof}

Let $\sigma^{(1)}$ be the sequence $\sigma_{1...4} \cdot \sigma_5[1, \sigma_5^{-1}[\mathbf8] ]$, and $\sigma^{(2)}$ be the rest of sequence in $\sigma$. The length of $\sigma_{1...4}$ is $30$ because of the two duplicates. So $|\sigma^{(1)}| \ge 34$ and $|\sigma^{(2)}| \le 17$.\\

\noindent \textbf{Claim 9:} Pick any two letters $a$ and $b$ from $\{\mathbf1,\mathbf2,\mathbf3,\mathbf4,\mathbf6\}$. One of the $3$-perms $\langle a, b, \mathbf7 \rangle$ or $\langle b, a, \mathbf7 \rangle$ is generated at the position of $\mathbf7$ in $\sigma_2$, i.e. $\sigma_2[-1]$.
\begin{proof} There is only at most one duplicate in $\sigma_1$ because $f(\sigma, 8) \ge 4$.\end{proof}

Appending the letters $\mathbf5$ and $\mathbf8$ to any of these $3$-perms from Claim 9 produces a $5$-perm generated at the last element of $\sigma^{(1)}$.\\

\noindent \textbf{Claim 10:} There are at least $13$ occurrences of the letters $\{\mathbf1,\mathbf2,\mathbf3,\mathbf4,\mathbf6\}$ in $\sigma^{(2)}$.
\begin{proof} Pick any two letters from $\{\mathbf1,\mathbf2,\mathbf3,\mathbf4,\mathbf6\}$ and the letters $\mathbf5$ and $\mathbf8$ to produce a $5$-perm generated at the end of $\sigma^{(1)}$. The rest of the three letters must form a supersequence over three letters in $\sigma^{(2)}$. So, there are at least three letters occurring at least thrice and two letters occurring at least twice in $\sigma^{(2)}$. \end{proof}

To derive the final contradiction, note that there are at least two occurrences of the letters $\mathbf5$ and $\mathbf7$ and one occurrence of $\mathbf8$ in $\sigma^{(2)}$. This gives an additional five elements on top of the $13$ from Claim 10. But it exceeds the maximum length of $17$. This completes the proof of $f(\sigma, \mathbf8) \ge 5$.

\end{proof}

Although we already established the letter $\mathbf8$ must occur in $\sigma_4$ from the previous lemma, that is based on the assumption that $f(\sigma, \mathbf8) = 4$. We need to prove it again with $f(\sigma, \mathbf8) = 5$ in the following lemma, but the argument is largely similar.

\begin{lemma} \label{lem14}
If $f(\sigma, \mathbf8) = 5$, the letter $\mathbf8$ occurs in $\mathbf8 \in \sigma_4$.
\end{lemma}
\begin{proof} 
Let $d$ be the number of duplicates in $\sigma_{2...3}$. By assumption, $f(\sigma, \mathbf8) = 5$, $f(\sigma, \mathbf7) = 5$ and $f(\sigma, \mathbf6) = 5$, we know $d$ can be $0$, $1$ or $2$.
\\

\noindent \textbf{Claim 1:} The letter $\mathbf8$ occurs after $5-d$ out of the five letters from the set $A \setminus \{\mathbf8, \mathbf7, \mathbf6\}$ in $\sigma_3$.
\begin{proof} If $\sigma_1[-2] \neq \mathbf 7$, then the $4$-perm $\langle \sigma_1[-2], \mathbf7, a, \mathbf8 \rangle$ cannot be generated at $4$ if $a$ occurs after $\mathbf8$ in $\sigma_4$. So assume $\sigma_1[-2] = \mathbf7$. If $\mathbf6$ is one of the duplicate letters in $\sigma_2$, then any two non-duplicates $a$ and $b$ cannot both occur after $\mathbf8$ in $\sigma_3$ because either $\langle \mathbf7, a, b, \mathbf8 \rangle$ or $\langle \mathbf7, b, a, \mathbf8 \rangle$ cannot be generated at $4$. Finally, if $\mathbf6$ is not a duplicate letter in $\sigma_2$, then for any non-duplicate $a$, either $\langle \mathbf7, \mathbf6, a, \mathbf8 \rangle$ or $\langle \mathbf7, a, \mathbf6, \mathbf8 \rangle$ cannot be generated at $4$. \end{proof}

Let $t$ be the last element of $\sigma_4$ which either does not occur in $\sigma_3$ or occurs before $\mathbf8$ in $\sigma_3$. This is so that whatever $3$-perm $\langle a, b, \mathbf8 \rangle$ is generated at $3$, then the $4$-perm $\langle a, b, \mathbf8, t \rangle$ will be generated at $4$. Let $i = \sigma^{-1}_4[t]$ be the position of letter $t$ in $\sigma_4$. Let $\sigma^{(1)} = \sigma_{1...3} \cdot \sigma_4[1,i]$ and $\sigma^{(2)} = \sigma_4[i+1,-1] \cdot \sigma_{5...8}$ . 
\\

\noindent \textbf{Claim 2:} For any two letters $a, b \in A \setminus \{\mathbf8, t\}$, one of the $4$-perms $\langle a, b, \mathbf8, t \rangle$ or $\langle b, a, \mathbf8, t \rangle$ is generated at $\sigma_4[i]$.
\begin{proof} 
One of the $2$-perms $\langle a, b \rangle$ or $\langle b, a \rangle$ must be generated at $2$ because $\sigma_1$ is $1$-complete and every letter occurs once. Assume it is $\langle a, b \rangle$. Then $\langle a, b, \mathbf 8\rangle$ is generated at $3$ because $\mathbf8$ does not occur in $\sigma_2$. Since $\mathbf8$ occurs after $t$ in $\sigma_3$, the $4$-perms $\langle a, b,\mathbf8, t \rangle$ must be generated as claimed.
\end{proof}

\noindent \textbf{Claim 3:} Suppose $\alpha$ is a $4$-perm, and let $C$ be the set $C = A \setminus \alpha$. If $\alpha$ is generated at $\sigma_4[i]$ for an integer $0 \le i \le |\sigma_4|$, then there is a letter of $C$ that occurs at least four times, two letters that occur at least three times and one letter that occurs at least twice in $\sigma^{(2)}$.
\begin{proof} 
Any $4$-perm $\alpha'$ with letters from $C$ has to be a subsequence of $\sigma^{(2)}$ because $\alpha \cdot \alpha'$ is a permutation, and its first part can be generated at $\sigma_4[i]$. So $\sigma^{(2)}$ is a supersequence over a set of four letters, which must have the claimed frequency distribution.
\end{proof}

\noindent \textbf{Claim 4:} There are three letters in the set $A \setminus \{\mathbf8, t\}$ that occur at least four times, two letters that occur at least three times and one letter that occurs at least twice in $\sigma^{(2)}$. \begin{proof} 
There are a total of six letters in the set $A \setminus \{\mathbf8, t\}$. Suppose there are two or less letters that occur at least four times. Let $C$ be a set that contains any four letters from the rest. By Claim 3, at least one letter of $C$ must occur at least four times, which contradicts the choice of the letters of $C$. A similar argument proves there are at least two letters that occur at least three times.
\end{proof}

\noindent \textbf{Claim 5:} The letter $\mathbf7$ occurs before the terminal letter $\mathbf5$ in $\sigma_4$.
\begin{proof} The proof is exactly the same as the previous claims. Assume $\mathbf7$ occurs after $\mathbf5$ in $\sigma_4$, then its position is at least $6$th in $\sigma_4$. So, there are $22 + d + 6$ elements at or before $\mathbf7$ in $\sigma$, and there are $23 - d$ many elements after $\mathbf7$. Any two letters $a, b \in A \setminus \{\mathbf8, \mathbf7\}$ can form a $4$-perm $\langle a, b, \mathbf8, \mathbf7 \rangle$ or $\langle b, a, \mathbf8, \mathbf7\rangle$ that is generated at the position of $\mathbf7$ in $\sigma_4$. As in Claim 4, there are three letters in the set $A \setminus \{\mathbf8, \mathbf7\}$ that occur at least four times, two letters that occur at least three times and one letter that occurs at least twice after the position of $\mathbf7$ in the sequence. The total number plus the four occurrences of $\mathbf8$ and $\mathbf7$ exceeds the length of the remaining sequence, which is a contradiction.
\end{proof}

So there is a total of at least twenty occurrences of elements from the set $A \setminus \{\mathbf8, t\}$ in $\sigma^{(2)}$. Since $\mathbf8$ has to occur five times in the entire supersequence, but it only occurs in the two segments $\sigma_1$ and $\sigma_2$, it must occur thrice in $\sigma_{5...8}$. Finally the letter $t$ must occur at least twice too. So, the length of $\sigma^{(2)}$ must be at least twenty five. Since $i$ is at least $6-d$, the length of $\sigma^{(1)}$ is at least $22 + d + 6 - d=28$, which contradicts the assumption on the length of $\sigma$.
\end{proof}

Next, we establish the frequencies of letters $\mathbf7$ and $\mathbf6$.

\begin{lemma} \label{lem15}
Given $f(\sigma, \mathbf8) = 5$, the frequency of $\mathbf7$ in $\sigma$ is $6$.
\end{lemma}
\begin{proof} Using the argument of Claim 1 of Lemma 13, we establish that $\mathbf7 \notin \sigma_3$ and $\mathbf7 \in \sigma_4$.  So, either $\mathbf7 \notin \sigma_5$ or $\mathbf7 \notin \sigma_6$ is true. We also know either $\mathbf8 \notin \sigma_5$ or $\mathbf8 \notin \sigma_6$. Without loss of generality, we assume $\mathbf8 \notin \sigma_5$ and $\mathbf7 \notin \sigma_6$. The proof is similar for the other cases. Since $f(\sigma, \mathbf8) = 5$, there is no duplicate in $\sigma_1$. 

The major steps of the proof are similar to Lemma 13. The following Cases outline the major steps to arrive at the conclusion. \\

\noindent \textbf{Case 1:} There is no duplicate from a letter of $\{\mathbf1,\mathbf2,\mathbf3,\mathbf4\}$ in $\sigma_2$ or $\sigma_3$. Since all letters of $\{\mathbf1,\mathbf2,\mathbf3,\mathbf4\}$ occur three times in $\sigma_{1...3}$, they cannot form a supersequence over four letters. Since there is a $4$-perm of these letters generated at $4$. The other four letters also cannot form a supersequence over four letters because they occurs at most three times in $\sigma_{4...8}$. \\

\noindent \textbf{Case 2:} There is a duplicate from a letter of $\{\mathbf1,\mathbf2,\mathbf3,\mathbf4\}$ in $\sigma_2$ or $\sigma_3$. Without loss of generality, assume $\mathbf3$ is a duplicate in $\sigma_2$. There is a $4$-perm from the letters of $\{\mathbf1,\mathbf2,\mathbf3,\sigma_3[-1]\}$ that is generated at the last position of $\sigma_3$. The rest of the four letters cannot form a supersequence over four letters after that position because they occur at most three times in $\sigma_{4...8}$.
\end{proof}


With the established results for letters $\mathbf7$ and $\mathbf8$, we get the frequency of $\mathbf6$ through removal operation.

\begin{corollary} \label{corollary3}
If $f(\sigma, \mathbf8) = 5$ and $f(\sigma, \mathbf7) = 6$, then $f(\sigma, \mathbf6) = 6$.
\end{corollary}
\begin{proof} Two successive removal operations will produce a supersequence over six letters where the letter $\mathbf6$ occurs $f(\sigma, \mathbf6) - 2$ times. By Theorem 2, this number is $f(\sigma, \mathbf6) - 2 \ge 4$. It also cannot be strictly greater than four, because otherwise, another removal operation will produce a supersequence over five letters that has length smaller than minimum $19$.
\end{proof}

Finally, we prove that the letter $\mathbf4$ at the tail end of the supersequence needs to be increased to generate some permutations.

\begin{lemma} \label{lem16}
If $f(\sigma, \mathbf8) = 5$, the frequency of $\mathbf4$ in $\sigma$ is $7$. \end{lemma}
\begin{proof} Suppose $f(\sigma, \mathbf4) = 6$. If $\mathbf8 \notin \sigma_5$, then the $5$-perm $\langle \mathbf3,\mathbf2,\mathbf1,\mathbf7,\mathbf6 \rangle$ is generated at $5$. The remaining three letters $\{\mathbf4,\mathbf5,\mathbf8\}$ all occur twice in $\sigma_{5...8}$. So, they cannot form a supersequence over three letters. If $\mathbf8 \notin \sigma_6$, then the $6$-perm $\langle \mathbf3,\mathbf2,\mathbf1,\mathbf7,\mathbf6,\mathbf5 \rangle$ is generated at $6$. The remaining two letters $\{\mathbf4,\mathbf8\}$ both occur once in $\sigma_{6...8}$. So they cannot form a supersequence over two letters.
\end{proof}

\begin{theorem} \label{th3}
The minimum length of a supersequence over eight letters is $52$. \end{theorem}
\begin{proof} If $f(\sigma, \mathbf8) = 5$, then $f(\sigma, \mathbf4) = 7$. So the length of $\sigma$ is $52$. If $f(\sigma, \mathbf8) = 6$, then a removal operation, which deletes $13$ elements, will produce a supersequence over seven letters with minimum length of $39$. So the length of $\sigma$ is also $52$.   \end{proof}

\section{Afterthought}

The lower bound we know today is very close to the best supersequence Tan \cite{Tan} constructed in 2022, but there is still a gap. By studying the minimum length of a supersequence over a set of eight letters, we attempt to analyze the internal structure of a supersequence through the segmentation and the frequency distribution of terminal letters. This knowledge will help facilitate narrowing the range of the lower bound asymptotically.  

A common theme of the paper is to seek a pivotal position within the supersequence, and use the frequencies of letters in both ends of the position to limit the lengths of different substrings. The restriction is further utilized to gain knowledge about the frequency distribution, thereby bootstrapping to the next stage of development. However, this strategy fails when generalized to larger sets. The bounds to the lengths become too large to be useful, and worse, the recursion through removal operation yields less information with large $n$.

To find the optimal solution and close the gap, we probably need to study skippable letters of each segment and use some higher degree quasi-symmetry similar to the quasi-palindrome property. This will enable restricting the possible letter frequencies not just from both ends, but also deeper within each segment.



\end{document}